\documentclass[12pt,a4paper]{article}
\usepackage[cp1251]{inputenc}
\usepackage[T2A]{fontenc}
\usepackage[english]{babel}

\usepackage{amsmath,amsfonts,amssymb,mathrsfs,amsopn,indentfirst,amscd}

\usepackage{graphicx}
\usepackage{amsthm}
\usepackage{hyperref}

\usepackage{caption}
\usepackage{subcaption}

\usepackage[mag=1000,a4paper,left=3cm,right=2.2cm,top=3.1cm,bottom=3.0cm]{geometry}

\usepackage{amssymb}
\usepackage{amsmath}
\newcommand{\ver}{{\rm ver}}
\newcommand{\vo}{{\rm vol}}

\newtheorem*{corollary*}{Corollary}
\newtheorem{theor}{Theorem}
\newtheorem{corol}{Corollary}

\begin{document}
\title{Interpolation by Linear Functions \\
on an $n$-Dimensional Ball
}
\author{Mikhail Nevskii\footnote{Department of Mathematics,
              P.G.~Demidov Yaroslavl State University, Sovetskaya str., 14, Yaroslavl, 150003, Russia 
              orcid.org/0000-0002-6392-7618 
              mnevsk55@yandex.ru}        \and
        Alexey Ukhalov \footnote{ Department of Mathematics, P.G.~Demidov Yaroslavl State University, Sovetskaya str., 14, Yaroslavl, 150003, Russia 
              orcid.org/0000-0001-6551-5118 
              alex-uhalov@yandex.ru}
              }
\date{May 7, 2019}
\maketitle

\begin{abstract}
By $B=B(x^{(0)};R)$ we denote
the Euclidean ball in ${\mathbb R}^n$ given by the inequality
$\|x-x^{(0)}\|\leq R$. Here $x^{(0)}\in{\mathbb R}^n, R>0$, 
$\|x\|:=\left(\sum_{i=1}^n x_i^2\right)^{1/2}$. 
We mean by  
$C(B)$ the space of~continuous functions
$f:B\to{\mathbb R}$ with the norm
$\|f\|_{C(B)}:=\max_{x\in B}|f(x)|$  and by
$\Pi_1\left({\mathbb R}^n\right)$ the set of polynomials in
$n$ variables of degree $\leq 1$, i.e., linear functions on ${\mathbb R}^n$.
Let $x^{(1)}, \ldots, x^{(n+1)}$
be the vertices of
$n$-dimensional nondegenerate simplex  $S\subset B$.
The interpolation projector   
 $P:C(B)\to \Pi_1({\mathbb R}^n)$ corresponding to 
$S$ is defined by the equalities 
$Pf\left(x^{(j)}\right)=f\left(x^{(j)}\right).$ 
We obtain the formula to compute the norm of $P$ as an operator
from $C(B)$ into $C(B)$ via $x^{(0)}$, $R$ and coefficients of 
basic Lagrange polynomials of $S$.  
In more details we study the case when $S$ is a regular
simplex inscribed into $B_n=B(0,1)$.
\medskip 

\noindent Keywords: $n$-dimensional simplex, $n$-dimensional ball, linear interpolation, projector, norm
\end{abstract}

\section{Introduction}\label{nev_s1}
We begin with basic definitions and notations.
Everywhere further $n\in{\mathbb N}.$ An element $x\in{\mathbb R}^n$ 
we present in the form $x=(x_1,\ldots,x_n).$ 
Denote by $e_1$, $\ldots$, $e_n$ the standard basis in ${\mathbb R}^n$.
Let us put $$\|x\|:=\sqrt{(x,x)}=\left(\sum\limits_{i=1}^n x_i^2\right)^{1/2},$$ 
$$B\left(x^{(0)};R\right):=\{x\in{\mathbb R}^n: \|x-x^{(0)}\|\leq R \} 
\quad \left(x^{(0)}\in {\mathbb R}^n,
R>0\right),$$ \ 
$$B_n:=B(0;1), \quad 
Q_n:=[0,1]^n, \quad
Q_n^\prime:=[-1,1]^n.$$

Everywhere further $S$ is a nondegenerate simplex in ${\mathbb R}^n$. 
By $c(S)$ denote the center of gravity of $S$. 
The notation $\tau S$ means the result of the homotety of $S$ with homothetic center $c(S)$
and  coefficient $\tau$. We denote $\ver(G)$ the set of vertices of the convex polyhedron $G$.
Let  $\Omega \subset{\mathbb R}^n$ be the closed bounded set. Consider the value
$\xi(\Omega;S):=\min \{\sigma\geq 1: \Omega\subset \sigma S\}$.
We call this number {\it the absorption index of $S$ with respect to $\Omega$}.
             
Let $x^{(j)}=\left(x_1^{(j)},\ldots,x_n^{(j)}\right),$ $1\leq j\leq n+1$ be 
the vertices of the simplex $S$.
{\it The vertex matrix}
$${\bf A} :=
\left( \begin{array}{cccc}
x_1^{(1)}&\ldots&x_n^{(1)}&1\\
x_1^{(2)}&\ldots&x_n^{(2)}&1\\
\vdots&\vdots&\vdots&\vdots\\
x_1^{(n+1)}&\ldots&x_n^{(n+1)}&1\\
\end{array}
\right)$$
is non-degenerate. Put ${\bf A}^{-1}$ $=(l_{ij})$. 
We call linear polinomials 
$\lambda_j(x)=
l_{1j}x_1+\ldots+
l_{nj}x_n+l_{n+1,j}$
{\it the basic Lagrange polynomials of S}.
Note that the coefficients of $\lambda_j$ form the columns of  ${\bf A}^{-1}$.
These polinomials have the property 
$\lambda_j\left(x^{(k)}\right)$ $=$ $\delta_j^k$.
The numbers $\lambda_j(x)$  are  barycentric coordinates of the point 
 $x\in{\mathbb R}^n$ with respect to $S$. 
 The simplex  $S$ is given by the system of linear inequalities
   $\lambda_j(x)\geq 0$. 
More on polinomials $\lambda_j$ see in [4, Ch.\,1].

Let us introduce the value 
$$\xi_n(\Omega):=\min \{ \xi(\Omega;S): \,
S \mbox{ --- $n$-dimensional simplex,} \,
\ver (S)\subset \Omega, \, \vo(S)\ne 0\},$$
$\xi_n:=\xi_n(Q_n).$
Various estimates of numbers $\xi_n$ were obtained in papers
[2],
[3],
[5],
[6],
[7],
[8],
[13] and in
monograph
[4]. 
Here let us note that $n\leq \xi_n< n+1$. 
The exact values of $\xi_n$ are known only for
$n=2$, $n=5$, $n=9$, and for infinite numer of such $n$ that there 
exists a Hadamard matrix of order $n+1$.
In all these cases, except $n=2$, holds $\xi_n=n,$ while $\xi_2=\frac{3\sqrt{5}}{5}+1=2.34\ldots$
Up to present time  $n=2$ is the only even $n$ for which the exact value of $\xi_n$ is known.
The number $\xi_n$ for a Euclidean ball, i.\,e., exact value of
$\xi_n(B_n)$, was found in [10]. 
In particular, for any  $S\subset B_n$ holds $\xi(B_n;S)\geq n$ with the equality only 
for a regular simplex $S$ inscribed into the border sphere.
Further, in such a case, we say that the simplex is inscribed into the ball.
Therefore, always $\xi_n(B_n)=n$.

Denote by $C(\Omega)$ the space of continuous functions  
$f:\Omega\to{\mathbb R}$ with the uniform 
norm $$\|f\|_{C(\Omega)}:=\max\limits_{x\in Q_n}|f(x)|.$$
By $\Pi_1\left({\mathbb R}^n\right)$ we mean the set of polynomials in $n$ variables 
of degree $\leq 1$ (or the set of linear functions in $n$ variables).

Let $x^{(j)}\in\Omega$, $1\leq j\leq n+1,$ be the vertices of  $n$-dimensional  non-degenerate
simplex $S$.  
We say that interpolation projector $P:C(\Omega)\to \Pi_1({\mathbb R}^n)$ 
corresponds to $S$ if nodes of interpolation 
coincide with the points $x^{(j)}$.
This projector is determined by the equalities
$Pf\left(x^{(j)}\right)=
f_j:=
f\left(x^{(j)}\right).$
The analogue of the Lagrange interpolation formula 
\begin{equation}\label{interp_Lagrange_formula}
Pf(x)=\sum\limits_{j=1}^{n+1}
f\left(x^{(j)}\right)\lambda_j(x) 
\end{equation}
holds true.
Denote by  $\|P\|_\Omega$ the norm of $P$ as an operator from $C(\Omega)$ into $C(\Omega)$. 
It follows from (\ref{interp_Lagrange_formula}) that
$$\|P\|_\Omega=\sup_{\|f\|_{C(\Omega)}=1} \|Pf\|_{C(\Omega)}=
\sup_{-1\leq f_j\leq 1} \max_{x\in\Omega} \sum_{j=1}^{n+1} f_j\lambda_j(x).$$
The expression $\sum f_j\lambda_j(x)$ is linear in $x$ and in
$f_1,\ldots,f_{n+1}$, therefore,
\begin{equation}\label{norm_P_intro_cepochka}
\|P\|_\Omega=\max_{f_j=\pm 1} \max_{x\in\Omega} \sum_{j=1}^{n+1}
f_j\lambda_j(x)
=\max_{x\in\Omega}\max_{f_j=\pm 1} \sum_{j=1}^{n+1}
f_j\lambda_j(x)
=\max_{x\in\Omega}\sum_{j=1}^{n+1}
|\lambda_j(x)|.
\end{equation}
If $\Omega$ is a convex polyhedron in ${\mathbb R}^n$ 
(for instance, if $\Omega=Q _n)$, we have simpler equality 
$$\|P\|_\Omega= \max_{x\in\ver(\Omega)}\sum_{j=1}^{n+1}
|\lambda_j(x)|.
$$

We say that the point $x\in\Omega$ is an 
{\it $1$-point of $\Omega$ with respect to $S,$} if for projector 
$P:C(\Omega)\to\Pi_1\left({\mathbb R}^n\right)$ 
with nodes in vertices of $S$ holds $\|P\|=\sum  \lambda_j(x)|$ 
and only one of numbers  $\lambda_j(x)$ is negative.
In [1, Theorem 2.1]
it was established that for any 
$P$ and corresponding  $S$ holds
\begin{equation}\label{nev_ksi_P_ineq}
\frac{n+1}{2n}\Bigl( \|P\|_\Omega-1\Bigr)+1\leq 
\xi(\Omega;S) \leq
\frac{n+1}{2}\Bigl( \|P\|_\Omega-1\Bigr)+1.
\end{equation}
If there exists an $1$-point in $\Omega$ with respect to $S$, then
the last inequality in (\ref{nev_ksi_P_ineq}) becomes an equality.
(The last statement was proved in [1]
in the equivalent form; the notion of $1$-vertex of the cube was introduced later
in [2].)

By $\theta_n(\Omega)$ denote the minimal value  $\|P\|_\Omega$
where  $x^{(j)}\in \Omega$.  It was proved by the first author that for numbers 
$\theta_n:=\theta_n(Q_n)$  the  asymptotic  equality $\theta_n\asymp \sqrt{n}$ takes place
(various estimates were systematized in [4]). 
Later on some estimates were improved by
M.\,V.~Nevskii and A.\,Yu.~Ukhalov, and by students of the second author (see
[5],
[6],
[9] and references in these works). 
By the present time the exact values of  $\theta_n$ are known only for $n=1, 2, 3$, and $7$. 
Specifically, $\theta_1=1,$ $\theta_2=\frac{2\sqrt{5}}{5}+1=1.89\ldots$, $\theta_3=2$,
$\theta_7=\frac{5}{2}$. For all these cases $$\xi_n=
\frac{n+1}{2}\Bigl( \theta_n-1\Bigr)+1.$$

In this paper we consider the case $\Omega=B_n$. 
Since for $S\subset B_n$ holds $\xi(B_n;S)\geq n,$ then for the 
corresponding projector $P$ the last inequality in (\ref{nev_ksi_P_ineq}) gives
\begin{equation}\label{3_minus_frac_bound_for_P}
\|P\|_{B_n}\geq 3 - \frac{4}{n+1}.
\end{equation}
So, always  $\theta_n(B_n)\geq 3-\frac{4}{n+1}$.

If for some interpolation projector
$P:C(B_n)\to\Pi_1\left({\mathbb R}^n\right)$ in  (\ref{3_minus_frac_bound_for_P})
takes place the equality, then the corresponding simplex $S$ is regular and inscribed into $B_n$. 
Indeed, in this case 
$$\frac{n+1}{2}\Bigl( \|P\|_{B_n}-1\Bigr)+1= n,$$
and  
(\ref{nev_ksi_P_ineq})  gives $\xi(B_n;S)\leq n.$ 
Hence $\xi(B_n;S)=n$. This is possible only for a regular simplex inscribed into $B_n$.  
Therefore, it follows from the equality $\theta_n(B_n)= 3-\frac{4}{n+1}$ that the simplex 
corresponding to the minimal projector has the described form.

If $S$ is a regular simplex inscribed into $B_n$ and there exists an $1$-point of $B_n$ with respect to $S$,
then the relation (\ref{3_minus_frac_bound_for_P}) is an equality.
When $S$ is not a regular or not inscribed into the ball, (\ref{3_minus_frac_bound_for_P}) is a strict inequality.
As it will be demonstrated, an $1$-point of the ball with respect to an inscribed regular simplex
exists only for $1\leq n \leq 4$. Beginning from  $n=5$ the equality 
in (\ref{3_minus_frac_bound_for_P}) doesn't hold true for all  $n$ and $P$,
therefore, for $n\geq 5$ holds $\theta_n(B_n)>3-\frac{4}{n+1}$.

This paper includes proofs of this and some other results for linear interpolation on the ball.
In particular we will obtain the exact formula for the norm of the interpolation projector 
with nodes in vertices of an inscribed regular simplex. Utilizing this formula we will 
show that $\theta_n(B_n)\leq \sqrt{n+1}$ and the equality here is possible 
only for $n=m^2-1$, $m\geq 2$.

\section{Reduction in the problem of minimal projector}\label{nev_s2}
An interpolation projector $P:C(B_n)\to \Pi_1({\mathbb R}^n)$ is called a {\it minimal projector}
if  $\|P\|_{B_n}=\theta_n(B_n)$. 
The existence of a minimal projector follows from continuity of $\|P\|$ as 
a function of nodes that is defined on closed bounded subset of 
${\mathbb R}^m,$ $m=n(n+1)$, given by the conditions
$x^{(j)}\in B_n,$ $\det {\bf A}\geq \varepsilon_n>0$. 
Let us show that the minimal projector can be determined by the set of 
nodes belonging to the bounding sphere $\|x\|=1$.

Assume $P$  is an interpolation projector with the nodes 
$x^{(j)}\in B_n$ and $\lambda_j$ are the basic Lagrange polynomials of the 
corresponding simplex $S$. 
Suppose not all $x^{(j)}$ belong to the bounding sphere.
We will construct a projector $P^\prime$ such that its number of nodes on the border of $B_n$ is greater by 1 than for $P$ and such that  $\|P^\prime\|\leq \|P\|$. 
Assume the node $x^{(i)}$ of $P$ lies strictly inside $B_n$. 
Let $z$ be the ortogonal projection of $x^{(i)}$ on the hyperplane $\lambda_i(x)=0$ and $y$ be the point of intersection of the sphere and the straight line $(zx^{(i)})$ in the same halfspace as $x^{(i)}$.
Hence
$$0<\sigma:=\frac{\|y-x^{(i)}\|}{\|z-x^{(i)}\|}<1.$$
Denote by $T$  the compression of ${\mathbb R}^n$ with coefficient $\sigma$  
with respect to hyperplane $\lambda_i(x)=0$ in the orthogonal direction. 
Consider the projector $P_1$ with the same nodes but acting on ellipsoid $T(B_n)$.  
The node $x^{(i)}$ belongs to the border of  $T(B_n)$ and 
the rest of nodes belong to the hyperplane $\lambda_i(x)=0$. 
So, the number of nodes on the border of  $T(B_n)$ in comparision with the number of original 
nodes on the border of $B_n$ will be greater by 1.
Now, let us consider the inverse transformation $T^{-1}$.
This nondegenerate affine transformation maps the ellipsoid $T(B_n)$ into the ball $B_n$ and the 
simplex $S$ into the simplex $S^\prime=T^{-1}(S)$.
Obviously, the $i$th vertex of $S^\prime$ coincides with $y$ while all other 
vertices of $S^\prime$ are the same as those of $S$.
 
Let $P^\prime$ be the interpolation projector with nodes in vertices of $S^\prime$ that again is being considered on $B_n$.
It remains to compare the norms of the projectors.
Since
$T(B_n)\subset B_n,$ we have
$$\|P_1\|_{T(B_n)}
=\max_{x\in T(B_n)}\sum_{j=1}^{n+1}
|\lambda_j(x)|\leq
\max_{x\in B_n}\sum_{j=1}^{n+1}
|\lambda_j(x)|=
\|P\|_{B_n}.$$
Because the spaces $C(T(B_n))$ and $C(B_n)$ are isometric, 
holds $\|P^\prime\|_{B_n}= \|P_1\|_{T(B_n)}$ .
Therefore,  $\|P^\prime\|_{B_n}\leq \|P\|_{B_n}.$ 

Applying the described procedure for necessary number of times, we can move all the nodes to the boundary of $B_n$ without increasing the projector norm.
It means that there exists a minimal projector with all the nodes 
on the sphere.
In particular, this property can be utilized for numerical minimization of the projector norm.

\section{The norm of projector for interpolation on a ball}\label{nev_s3}

Let $B=B(x^{(0)};R)$ and let $P:C(B)\to \Pi_1({\mathbb R}^n)$ be 
the interpolation projector with the nodes  $x^{(j)}\in B$. 
Suppose  $S$ is the simplex with vertices  $x^{(j)}$ and 
$\lambda_j(x)=l_{1j}x_1+ \ldots+l_{nj}x_n+l_{n+1,j}$ are 
the basic Lagrange polynomials of $S$. 
In this case we have
\begin{equation}\label{norm_P_B_Lagrange}
\|P\|_{B}=
\max_{x\in B}\sum_{j=1}^{n+1}
|\lambda_j(x)|
\end{equation}
(see (\ref{norm_P_intro_cepochka}) for $\Omega=B$). 
Let us obtain another formula for the projector norm.

 \begin{theor}\label{theor_P_norm_max_f_j} The following equality holds:
$$\|P\|_B=
\max\limits_{f_j=\pm 1} \left[
R \left(\sum_{i=1}^n\left(\sum_{j=1}^{n+1} f_jl_{ij}\right)^2\right)^{1/2}
+\left|\sum_{j=1}^{n+1}f_j\left(\sum_{i=1}^n l_{ij}x_i^{(0)}+l_{n+1,j}\right)\right|
\right]=$$
\begin{equation}\label{norm_P_B_formula}
=
\max\limits_{f_j=\pm 1} \left[
R \left(\sum_{i=1}^n\left(\sum_{j=1}^{n+1} f_jl_{ij}\right)^2\right)^{1/2}
+\left|\sum_{j=1}^{n+1}f_j\lambda_j(x^{(0)})\right|
\right].
\end{equation}
\end{theor}

\begin{proof}[Proof]  Let us write the double maximum in the formula for $ \| P \| _B $ in a different order than while obtaining  (\ref{norm_P_B_Lagrange}):
\begin{equation}\label{normP_modulus}
\|P\|_B=\max_{f_j=\pm 1} \max_{x\in B} \sum_{j=1}^{n+1}
f_j\lambda_j(x)=
\max_{f_j=\pm 1} \max_{x\in B} \left|\sum_{j=1}^{n+1}
f_j\lambda_j(x)\right|.
\end{equation}  
For fixed $f_j$, the function
$$\Lambda(x):=\sum_{j=1}^{n+1} f_j\lambda_j(x)=
\sum_{i=1}^n\left(\sum_{j=1}^{n+1}f_jl_{ij}\right) x_i
+\sum_{j=1}^{n+1}f_jl_{n+1,j}$$
is linear in х. 
Maximum of $|\Lambda(x)|$ on the ball $B=B(x^{(0)};R)$  
is being achieved at one of two points of intersection of the sphere
  $\|x-x^{(0)}\|=R$ 
with the straight line passing through $x^{(0)}$ in
directing of the vector 
$v=(v_1,\ldots,v_n),$ where $v_i=\sum\limits_{j=1}^{n+1}f_jl_{ij}$. 
These points are
$$x_+=x^{(0)}+\frac{R}{\|v\|}v, \quad 
x_-=x^{(0)}-\frac{R}{\|v\|}v.$$ 
The function $L(x):=\Lambda(x)-\Lambda(0)$ is additive and homogeneous, 
therefore,
$$L(x_+)=L\left(x^{(0)}+\frac{R}{\|v\|}v\right)=
L\left(x^{(0)}\right)+
\frac{R}{\|v\|}L(v),$$
$$\Lambda(x_+)=
L(x_+)+\Lambda(0)=\Lambda\left(x^{(0)}\right)+ \frac{R}{\|v\|}
\Bigl[\Lambda(v)-\Lambda(0)\Bigr].$$
Note that
$$\Lambda(v)-\Lambda(0)=
\sum_{i=1}^n\left(\sum_{j=1}^{n+1}f_jl_{ij}\right) v_i= 
\sum_{i=1}^n\left(\sum_{j=1}^{n+1}f_jl_{ij}\right)^2=\|v\|^2,$$
hence,
$\Lambda(x_+)=\Lambda\left(x^{(0)}\right)+ R\|v\|.$ 
Analogously, 
$\Lambda(x_-)=\Lambda\left(x^{(0)}\right)-R\|v\|.$ 
Thus,
$$\max_{x\in B} |\Lambda(x)|=\max(|\Lambda(x_+)|,|\Lambda(x_-)|)=
R\|v\|+ 
\left|\Lambda\left(x^{(0)}\right)\right|.$$
Now, the equality  
(\ref{normP_modulus}) gives
$$\|P\|_B=\max_{f_j=\pm 1} \max_{x\in B} |\Lambda(x)|=
\max_{f_j=\pm 1} \left[ R\|v\|+ 
\left|\Lambda\left(x^{(0)}\right)\right|\right].$$
Taking into account that 
$$\|v\|=
\left(\sum_{i=1}^n\left(\sum_{j=1}^{n+1} f_jl_{ij}\right)^2\right)^{1/2},$$
we obtain
$$\Lambda\left(x^{(0)}\right)=
\sum_{j=1}^{n+1}f_j\lambda_j(x^{(0)})=
\sum_{j=1}^{n+1}f_j\left(\sum_{i=1}^n l_{ij}x_i^{(0)}+l_{n+1,j}\right).$$
The theorem is proved. \end{proof}

In the case when the center of gravity of the simplex coincides with the center of the ball formula (\ref{norm_P_B_formula}) becomes more simple.

\begin{corol}\label{corol_norm_P_c_eq_x0}
If $c(S)=x^{(0)}$, then
\begin{equation}\label{norm_P_B_formula_c_eq_x0}
\|P\|_B=
\max\limits_{f_j=\pm 1} \left[
R \left(\sum_{i=1}^n\left(\sum_{j=1}^{n+1} f_jl_{ij}\right)^2\right)^{1/2}
+\frac{1}{n+1}\left|\sum_{j=1}^{n+1}f_j\right|
\right].
\end{equation}
\end{corol}

\begin{proof}[Proof] The numbers $\lambda_j(x^{(0)})$ are the barycentric coordinates of  $x^{(0)}$ with respect to $S$. 
Since,
$$x^{(0)}=c(S)=\frac{1}{n+1}\sum_{j=1}^{n+1} x^{(j)}$$
we have $\lambda_j(x^{(0)})=\frac{1}{n+1}$ 
for all $j$. Hence, we obtain 
(\ref{norm_P_B_formula_c_eq_x0}) from
(\ref{norm_P_B_formula}).
\end{proof}

\section{The projector norm for a regular simplex \\inscribed into the ball}\label{nev_s4}

Suppose $S$ is a regular simplex inscribed into the $n$-dimensional ball
$B=B(x^{(0)};R)$ and
$P:C(B)\to \Pi_1({\mathbb R}^n)$ is the corresponding  interpolation projector.
Clearly, $\|P\|_B$ does not depend on the center $x^{(0)}$ and 
the radius $R$ of the ball and  on the choice of a regular simplex 
inscribed into that ball. 
In other words, $\|P\|_B$ is a function of dimension $n$. 
In this section, we will obtain exact representation of this function 
and establish its estimates.

For $0\leq t\leq n+1$, consider the function
\begin{equation}\label{psi_function_modulus}
\psi(t):=\frac{2\sqrt{n}}{n+1}\Bigl(t(n+1-t)\Bigr)^{1/2}+
\left|1-\frac{2t}{n+1}\right|.
\end{equation}
Denote
$a:=\left\lfloor\frac{n+1}{2}-\frac{\sqrt{n+1}}{2}\right\rfloor,$
where $\lfloor s \rfloor$ is the integer part of $s$.

 \begin{theor}\label{norm_P_regular} The following relations hold:
\begin{equation}\label{norm_P_reg_formula}
\|P\|_B=\max\{\psi(a),\psi(a+1)\},
\end{equation}
\begin{equation}\label{norm_P_reg_ineqs}
\sqrt{n}\leq \|P\|_B \leq \sqrt{n+1}.
\end{equation}
The equality $\|P\|_B=\sqrt{n}$ is true only for $n=1$.
The equality $\|P\|_B=\sqrt{n+1}$ holds if and only if 
$\sqrt{n+1}$ is an integer. 

\end{theor}

\begin{proof}[Proof] We first prove (\ref{norm_P_reg_formula}).
If $n=1$, then
$\psi(t)=\sqrt{t(2-t)}+|1-t|$, $a=0$, $\psi(a)=\psi(a+1)=1.$
Since $\|P\|_B=1$, the equality (\ref{norm_P_reg_formula}) is true.

Let $n\geq 2$. 
Consider the simplex $S$
with vertices
$$x^{(1)}=e_1, \ \  \ldots, \ \ x^{(n)}=e_n, \quad
x^{(n+1)}=
\left(\frac{1-\sqrt{n+1}}{n},\ldots,\frac{1-\sqrt{n+1}}{n}\right).
$$

This simplex is regular because  the length of any edge is equal to 
$\sqrt{2}$. 
Simplex $S$ is inscribed into the ball
$B=B(x^{(0)};R)$, where
$$x^{(0)}=c(S)=\left(\frac{1-\sqrt{\frac{1}{n+1}}}{n},\ldots,
\frac{1-\sqrt{\frac{1}{n+1}}}{n}\right), 
\quad R=\sqrt{\frac{n}{n+1}}.$$ 
Note that the $(n+1)$th vertex of $S$ is obtained by shifting of the zero vertex of the simplex  $x_i\geq 0$, $\sum x_i\leq 1$ in the direction from the hyperplane $\sum x_i=1$. 
It is important that $S$ is invariant with respect to changing the order of coordinates.
It is enough to find $\|P\|_B$ for this simplex.

The corresponding matrices ${\bf A}$ and ${\bf A}^{-1}$ are
\begin{equation}\label{A_A_minus_1_for_regular_S}
{\bf A}=
\left( \begin{array}{ccccc}
1&0&\ldots&0&1\\
0&1&\ldots&0&1\\
\vdots&\vdots&\vdots&\vdots&\vdots\\
0&0&\ldots&1&1\\
-\tau&-\tau&\ldots&-\tau&1\\
\end{array}
\right), \quad
{\bf A}^{-1} =
\frac{1}{\sqrt{n+1}}
\left( \begin{array}{ccccc}
\sigma&-\tau&\ldots&-\tau&-1\\
-\tau&\sigma&\ldots&-\tau&-1\\
\vdots&\vdots&\vdots&\vdots&\vdots\\
-\tau&-\tau&\ldots&\sigma&-1\\
\tau&\tau&\ldots&\tau&1\\
\end{array}
\right).
\end{equation}
Here
\begin{equation}\label{sigma_tau_formulae}
\sigma:=\frac{(n-1)\sqrt{n+1}+1}{n}, \quad \tau:=\frac{\sqrt{n+1}-1}{n}.
\end{equation}
Since $c(S)=x^{(0)}$, for calculation $\|P\|_B$ we can use 
(\ref{norm_P_B_formula_c_eq_x0}) with $R=\sqrt{\frac{n}{n+1}}$:
$$
\|P\|_B=
\max\limits_{f_j=\pm 1} \left[
\sqrt{\frac{n}{n+1}} 
 \left(\sum_{i=1}^n\left(\sum_{j=1}^{n+1} f_jl_{ij}\right)^2\right)^{1/2}
+\frac{1}{n+1}\left|\sum_{j=1}^{n+1}f_j\right|
\right].
$$
Here $l_{ij}$ are the elements of ${\bf A}^{-1}$ (see (\ref{A_A_minus_1_for_regular_S})).

Denote by $k$ the number of $f_j$ equal to $-1$. 
Then the number of $f_j$ equal to $1$ is $n+1-k$. 
The simplex $S$ does not change with renumerating of coordinates, so,
we may assume that 
$f_1=\ldots=f_k=-1,$ $f_{k+1}=\ldots=f_{n+1}=1$. 
Since the function being maximized does not change when the signs of 
$f_j$ change simultaneously, we can consider only the interval $1 \leq k \leq \frac{n + 1}{2}$.
Thus,
\begin{equation}\label{P_norm_f_j_long}
\|P\|_B=
\max\limits_{1\leq k\leq \frac{n+1}{2}} \left[
\sqrt{\frac{n}{n+1}} 
 \left(\sum_{i=1}^n\left(\sum_{j=1}^{n+1} f_jl_{ij}\right)^2\right)^{1/2}
+\frac{n+1-2k}{n+1}
\right],
\end{equation}
where $f_j=-1$ for $1\leq j\leq k$ and $f_j=1$ for all other $j$. 
The number $n+1-2k$ is equal to the sum under the sign of absolute value.
Taking into account the multiplier $\frac{1}{\sqrt{n+1}}$ in the equality for 
${\bf A}^{-1}$, let us rewrite the value
$$W:=(n+1)\sum_{i=1}^{n}
\left(\sum_{j=1}^{n+1}f_jl_{ij}\right)^2.$$
Unilizing the explicit expressions for $l_{ij}$ we present this sum in the form
$$W=
(n+1)\sum_{i=1}^{k}\left(\sum_{j=1}^{n+1}f_jl_{ij}\right)^2+
(n+1)\sum_{i=k+1}^{n}\left(\sum_{j=1}^{n+1}f_jl_{ij}\right)^2
=W_1+W_2.$$
From 
(\ref{A_A_minus_1_for_regular_S}) and the described distribution of values $f_j$, we get
$$W_1=\sum_{i=1}^k (-\sigma+(k-1)\tau-(n-k)\tau-1)^2=k(2k\tau-\alpha)^2,$$
$$W_2=\sum_{i=k+1}^n (k\tau+\sigma-(n-1-k)\tau-1)^2=(n-k)(2k\tau+\beta)^2.$$
Here $\alpha=\sigma+(n+1)\tau+1$, $\beta=\sigma-(n-1)\tau-1$. 
It follows from (\ref{sigma_tau_formulae}) that $\alpha=2\sqrt{n+1},$ $\beta=0$, therefore,
$$W=
4k(k\tau-\sqrt{n+1})^2+(n-k)\cdot4k^2\tau^2=k^2(-8\sqrt{n+1}\tau+4n\tau^2)+4k(n+1)=$$
$$=-4k^2+4k(n+1)=4k(n+1-k).$$ 
This gives
$$
\|P\|_B=
\max\limits_{1\leq k\leq \frac{n+1}{2}} \left[
\sqrt{\frac{n}{n+1}} 
 \left(\frac{1}{n+1}W\right)^{1/2}
+\frac{n+1-2k}{n+1}\right]=$$
$$=\max\limits_{1\leq k\leq \frac{n+1}{2}} \left[
\frac{2\sqrt{n}}{n+1} 
\bigl(k(n+1-k)\bigr)^{1/2}+1-\frac{2k}{n+1}\right].$$
Recalling (\ref{psi_function_modulus}), we obtain
\begin{equation}\label{norm_P_eq_max_psi}
\|P\|_B=
\max\limits_{1\leq k\leq \frac{n+1}{2}} \psi(k).
\end{equation}
It remains to show that the last maximum is equal to the largest of numbers
$\psi(a)$ and $\psi(a+1)$, where 
$a=\left\lfloor\frac{n+1}{2}-\frac{\sqrt{n+1}}{2}\right\rfloor.$ 
To do this, we analyze the behavior of $\psi(t)$ over the whole interval $ [0, n + 1] $.

The function
$$\psi(t)=\frac{2\sqrt{n}}{n+1}\Bigl(t(n+1-t)\Bigr)^{1/2}+
\left|1-\frac{2t}{n+1}\right|, \quad 0\leq t\leq n+1,$$
has following properties:
$\psi(t)>0$,
$\psi(0)=\psi(n+1)=1,$ 
$\psi\left(\frac{n+1}{2}\right)=\sqrt{n}$,
$\psi(n+1-t)=\psi(t)$. 
The graph of $\psi(t)$ is symmetric with respect to the straight line $t=\frac{n+1}{2}$.
On each half of $[0,n+1]$ function
$\psi(t)$ is concave as a sum of two concave functions.
Indeed, for $0\leq t\leq \frac{n+1}{2}$
$$\psi(t)=\varphi_1(t)+\varphi_2(t), \quad
\varphi_1(t):=\frac{2\sqrt{n}}{n+1}\Bigl(t(n+1-t)\Bigr)^{1/2}, \quad
\varphi_2(t):=1-\frac{2t}{n+1},$$
where $\varphi_1(t)$ is concave as a superposition of the concave function 
$t(n+1-t)$ and the increasing concave function $\sqrt{t}$, while 
$\varphi_2(t)$ is a linear function. 
The derivation $\psi^\prime(t)$ is equal to zero only in two points
$$t_-:=\frac{n+1}{2}-\frac{\sqrt{n+1}}{2}, \quad  
t_+:=\frac{n+1}{2}+\frac{\sqrt{n+1}}{2}.$$  

\begin{figure}[h!]
 \centering
\includegraphics[width=16cm]{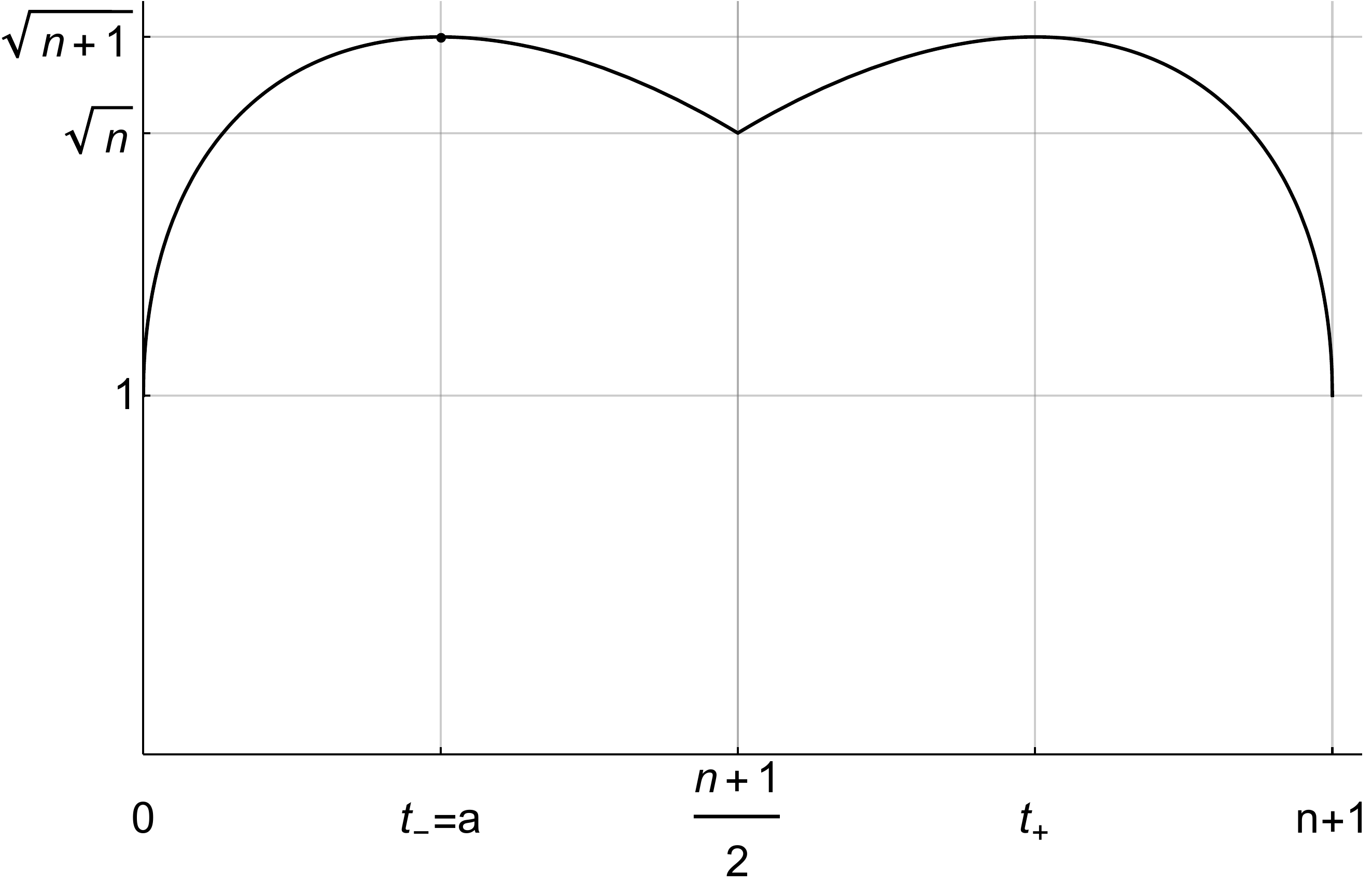}
\captionsetup{justification=centering}
\caption{Graph of $\psi(t)$ for $n=3$. Here $n+1=4$, $t_-=a=1$}
\label{fig:nev_ukl_n3}
\end{figure}
\begin{figure}[h!]
 \centering
\includegraphics[width=16cm]{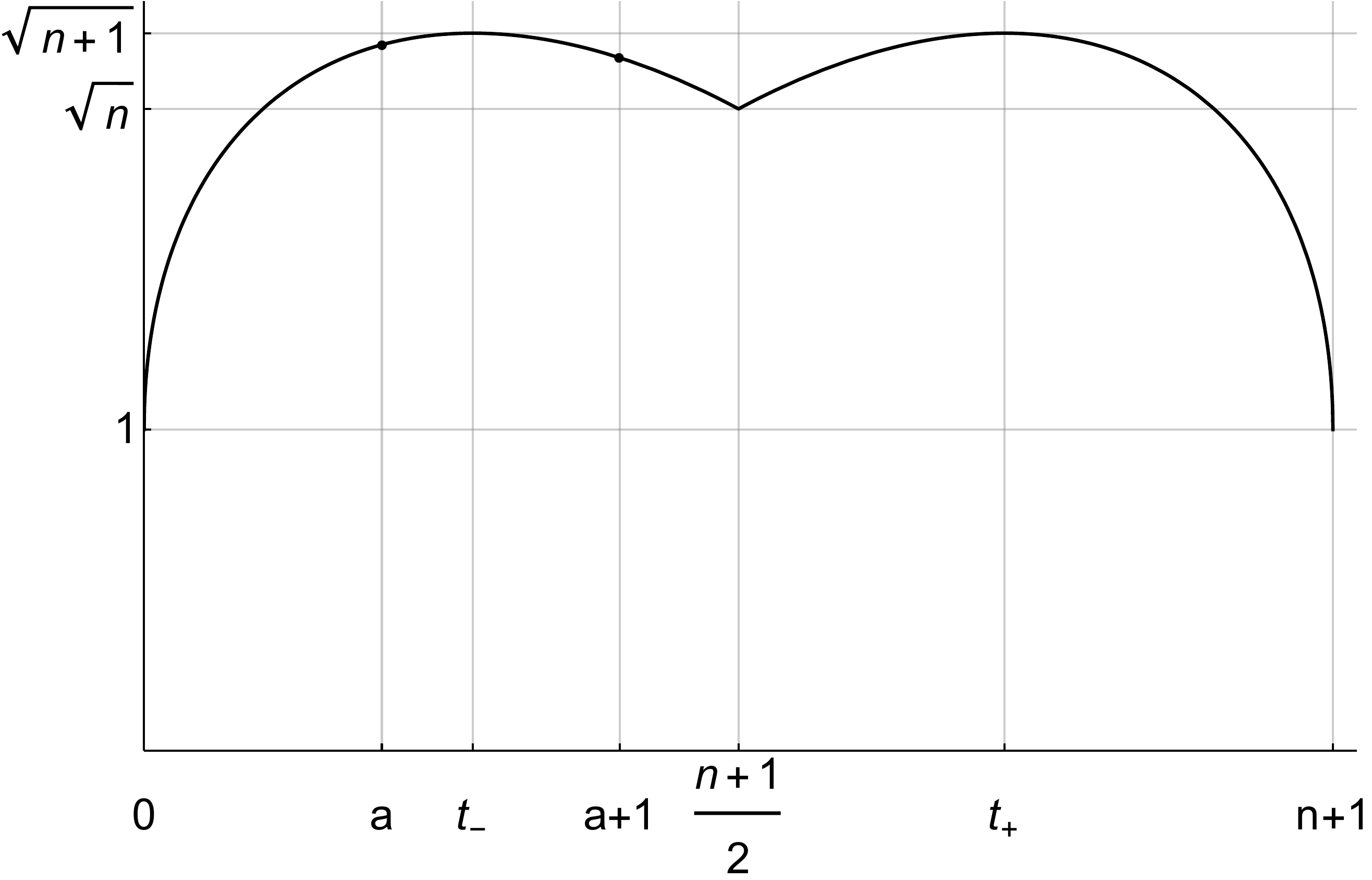}
\captionsetup{justification=centering}
\caption{Graph of $\psi(t)$ for $n=4$.  Here $n+1=5$, $a=1$, $t_-= \frac{ 5-\sqrt{5}}{2}$}
\label{fig:nev_ukl_n4}
\end{figure}
\begin{figure}[h!]
 \centering
\includegraphics[width=16cm]{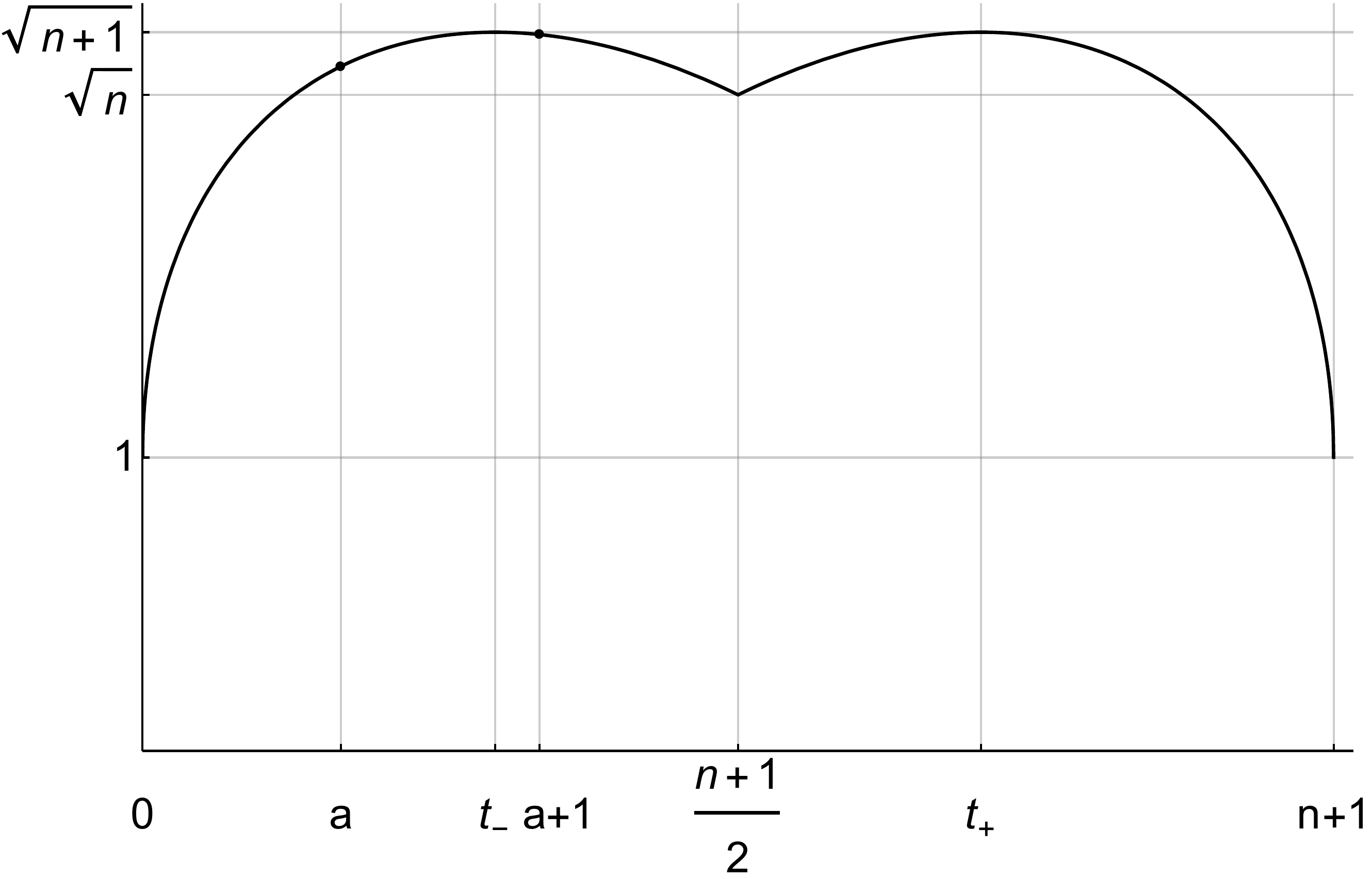}
\captionsetup{justification=centering}
\caption{Graph of  $\psi(t)$ for $n=5$. Here $n+1=5$, $a=1$, $t_-= \frac{6-\sqrt{6}}{2}$}
\label{fig:nev_ukl_n5}
\end{figure}
\begin{figure}[h!]
 \centering
\includegraphics[width=16cm]{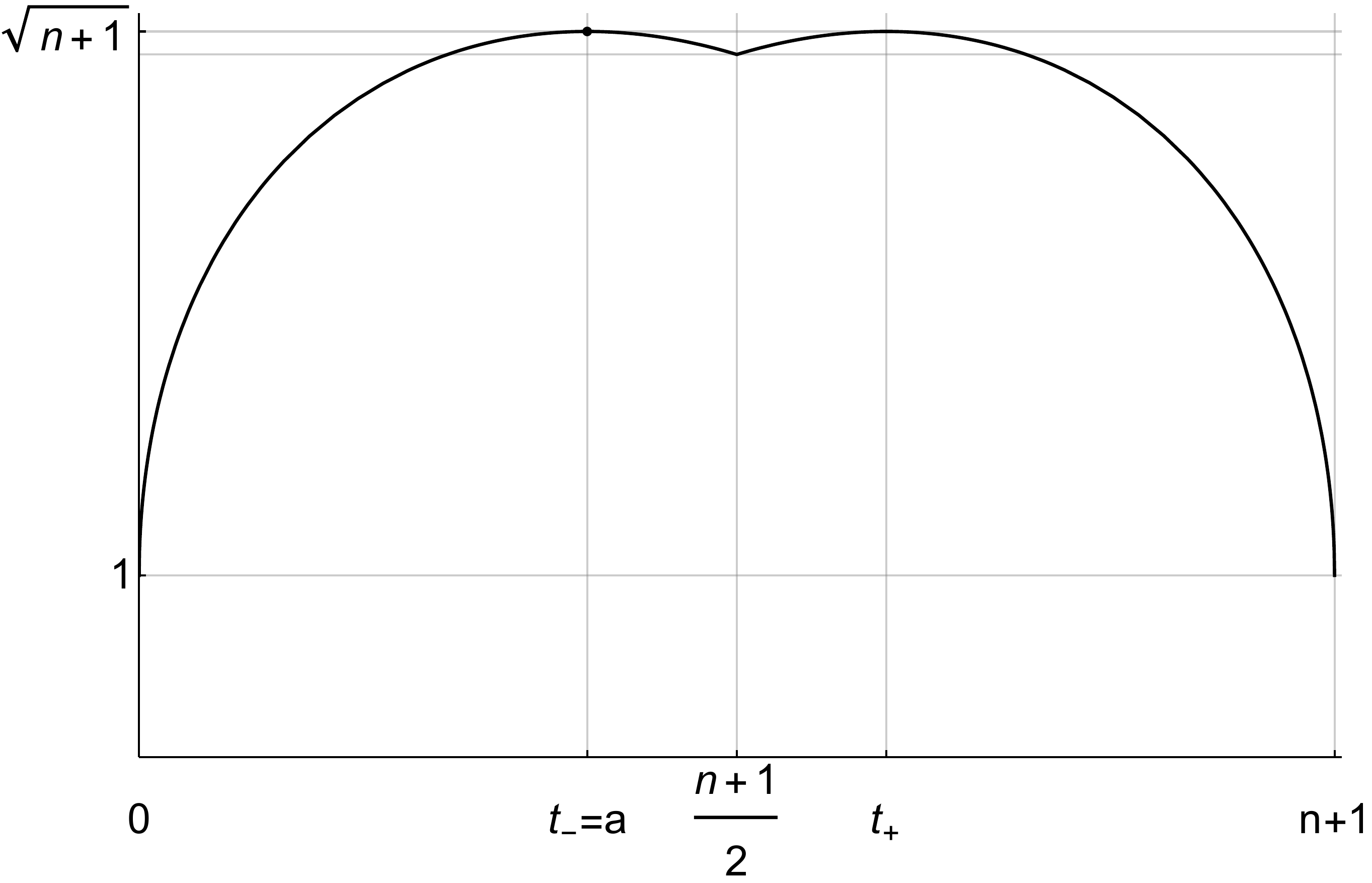}
\captionsetup{justification=centering}
\caption{Graph of $\psi(t)$ for $n=15$. Here $n+1=16$, $t_-=a=6$}
\label{fig:nev_ukl_n15}
\end{figure}

These points lie strictly inside in
$\left[0,\frac{n+1}{2}\right]$ and
$\left[\frac{n+1}{2},n+1\right]$ correspondingly.  
From the concavity of $\psi(t)$ on each of these segments,
it follows that
$$\max_{0\leq \psi(t)\leq n+1} \psi(t)=\psi(t_-)=\psi(t_+)=\sqrt{n+1}.$$
Moreover, $ t_-$ and $t _+ $ are the only maximum points on the left and right halves of $[0,n+1]$.
Hence, $\psi(t)$ increases for $0\leq t\leq t_-$ and decreases for $t_-\leq t\leq \frac{n+1}{2}$. 
On the left half $\psi(t)$ behaves symmetrically: it increases for  $\frac{n+1}{2}\leq t\leq t_+$
and decreases for $t_+\leq t\leq n+1.$

Let us now consider only the segment $\left[0,\frac{n+1}{2}\right]$.
Since $a=\left\lfloor t_-\right\rfloor$, then always $a\leq t_-<a+1$. 
Let $k$ be a whole number, $1\leq k\leq \frac{n+1}{2}$. 
If $k<a$, then $\psi(k)<\psi(a)$, while if  $k>a+1$, then $\psi(k)<\psi(a+1)$. 
Taking into account
(\ref{norm_P_eq_max_psi}), 
we have
$$\|P\|_B=
\max\limits_{1\leq k\leq \frac{n+1}{2}} \psi(k)=\max \{\psi(a),\psi(a+1)\}.
$$
The equality (\ref{norm_P_reg_formula}) is proved.  

We now turn to inequalities (\ref{norm_P_reg_ineqs}). We have already obtained the right one:
$$\|P\|_B=\max_{k\in {\mathbb Z}: \,1\leq k\leq\frac{n+1}{2}} \psi(k) \leq
\max_{1\leq t\leq\frac{n+1}{2}} \psi(t)=
\psi(t_-)=\sqrt{n+1}.$$
Let us describe such dimensions $n$ that $\|P\|_B=\sqrt{n+1}$.
These dimensions 
are characterized by the fact that the non-strict inequality in the last relation
becomes an equality, i.\,e., the number $t_-$ is integer.

Note that both $t_-$ and $t_+$  are integer if and only if $\sqrt{n+1}$  is an integer. 
Assume for example 
$$t_-:=\frac{n+1}{2}-\frac{\sqrt{n+1}}{2}=d\in {\mathbb Z}.$$
Then $\sqrt{n+1}=n+1-2d$ is an integer. On the contrary, if $\sqrt{n+1}=m\in{\mathbb Z}$,
then $m(m-1)$ is an even number and $t_-=\frac{m(m-1)}{2}$ is an integer. 
The statement for  $t_+$ can be proved similarly. Also it can be deduced from the previous statement, since $t_+-t_-=\sqrt{n+1}$.

We have obtained that for any dimension of the form  $n=m^2-1$, $m$
 is an integer, and only in these situations, the number $t_-$
is integer.  Consequently,
$$\|P\|_B=\max_{k\in {\mathbb Z}: \,1\leq k\leq\frac{n+1}{2}} \psi(k)
=\max_{1\leq t\leq\frac{n+1}{2}} \psi(t)=
\psi(t_-)=\sqrt{n+1}=m.$$
These equalities are equivalent to (\ref{norm_P_reg_formula}), 
since in the cases considered
$a=\left\lfloor t_-\right\rfloor=t_-$ and $\|P\|_B=\psi(a)$.
For all other $n$, holds $\|P\|<\sqrt{n+1}.$ 
Indeed, if $n \neq m^2-1$, then $a<t_-<a+1$  and the maximum in
(\ref{norm_P_eq_max_psi}) is reached either for $k=a$ or for $k=a+1$.
For any $n\ne m^2-1$ the norm of $P$, i.e., the maximum $\psi(k)$ for integer 
$k\in \left[1,\frac{n+1}{2}\right]$, is strictly less than maximum of $\psi(t)$ over all this interval.

It remains to show that always $\|P\|_B\geq \sqrt{n}$ and for $n>1$ 
this equality is a strict one.
If $n>3$, then $\sqrt{n+1}>2$, hence,
$$a+1=
\left\lfloor\frac{n+1}{2}-\frac{\sqrt{n+1}}{2}\right\rfloor+1\leq
\frac{n+1}{2}-\frac{\sqrt{n+1}}{2}+1<\frac{n+1}{2}.$$
From (\ref{norm_P_reg_formula}) and properties of $\psi(t)$, we get
$$\|P\|_B=\max \{\psi(a),\psi(a+1)\}\geq \psi(a+1)>\psi\left(\frac{n+1}{2}\right)=
\sqrt{n} \qquad(n>3).$$
For $n=2$ and $n=3$ also holds $\|P\|_B>\sqrt{n}$. 
Therefore, $\|P\|=\sqrt{n}$ only for $n=1$.
The theorem is proved.
\end{proof}

If $n+1$  is an Hadamard number, i.\,e., 
there exists an Hadamard matrix of order $n+1$ 
(
see [11], [12]), 
the estimate $\|P\|_B\geq c\sqrt{n}$  can 
be obtained from previous results of the first author. 
In 2006, he have proved that for all $n$
$$\theta_n(Q)\geq\chi_n^{-1}\left(\frac{1}{\nu_n}\right)>
\frac{1}{e}\sqrt{n-1}.
$$
Here $Q$ is an $n$-dimensional cube, 
$\chi_n$ is the normalized Legendre polynomial of degree $n$ and 
$\nu_n$ is the maximum volume of a simplex contained in the unit cube $Q_n=[0,1]^n$.
For references, definitions and proofs, see [4].
Suppose additionally that $n+1$ is an Hadamard number. 
Then there exists a regular simplex with vertices in the vertices of a cube (see [12]). 
Let $Q$ be a cube inscribed in the Euclidean ball $B$.
Then this regular simplex will be inscribed also in the ball.
Since $Q\subset B$, for the corresponding projector $P$ we have
$$\|P\|_B\geq \|P\|_Q\geq
\theta_n(Q)\geq\chi_n^{-1}\left(\frac{1}{\nu_n}\right)>
\frac{1}{e}\sqrt{n-1}.
$$

\section{The exact values $\theta_n(B_n)$ for $1\leq n\leq 4$}\label{nev_s5}

Utilizing Theorem \ref{norm_P_regular} and properties noted in Introduction, it is possible
to obtain the exact values of $\theta_n(B_n)$ and to describe
minimal projectors for $n=1,2,3,4$.
Of course, $\theta_1 (B_1) = 1 $, but this case also fits into the general scheme.

\begin{theor}\label{theta_1_2_3_4_theorem} The following equalities hold:
\begin{equation}\label{theta_1234_equalities}
\theta_1(B_1)=1, \quad
\theta_2(B_2)=\frac{5}{3}, \quad
\theta_3(B_3)=2, \quad
\theta_4(B_4)=\frac{11}{5}.
\end{equation}
For $1\leq n\leq 4$ the equality $\|P\|_{B_n}=\theta_n(B_n)$ holds only for projectors
corresponding to regular simpleces inscribed into $B_n$.
\end{theor}

\begin{proof}[Proof] 
Let us calculate the norm of the projector $P:C(B_n)\to \Pi_1({\mathbb R}^n)$
corresponding to a regular simplex inscribed into $B_n$.
By Theorem \ref{norm_P_regular},
$$
\|P\|_{B_n}=\max\{\psi(a),\psi(a+1)\},
$$  
where $\psi(t)$ is defined by (\ref{psi_function_modulus}) and 
$a=\left\lfloor\frac{n+1}{2}-\frac{\sqrt{n+1}}{2}\right\rfloor.$
For $n=1$ we have
$\psi(t)=\sqrt{t(2-t)}+|1-t|$, $a=0$, $\psi(a)=\psi(a+1)=1$ and
$\|P\|_{B_1}=1$.
For $n=2$ holds true 
$\psi(t)=\frac{2\sqrt{2}}{3}\sqrt{t(3-t)}+\left|1-\frac{2t}{3}\right|$, 
$a=0$, $\psi(a)=1$, $\psi(a+1)=\frac{5}{3}$ and
$\|P\|_{B_2}=\psi(a+1)=\frac{5}{3}$.
If  $n=3$, then
$\psi(t)=\frac{\sqrt{3}}{2}\sqrt{t(4-t)}+\left|1-\frac{t}{2}\right|$, 
$a=1$, $\psi(a)=\psi(a+1)=2$, hence, 
$\|P\|_{B_3}=2$.
Finally, for $ n = 4 $, we have to take
$\psi(t)=\frac{4}{5}\sqrt{t(5-t)}+\left|1-\frac{2t}{5}\right|$, 
$a=1$, $\psi(a)=\frac{11}{5}$, $\psi(a+1)=\frac{1}{5}\bigl(1+4\sqrt{6}\bigr)$, so,
$\|P\|_{B_4}=\psi(a)=\frac{11}{5}$.

Now note that for each $n=1, 2, 3, 4$ holds
$$\|P\|_{B_n}=3-\frac{4}{n+1}.$$
For any dimension, $\theta_n(B_n)\geq 3-\frac{4}{n+1}$. Therefore, if $1\leq n\leq 4$, then
\begin{equation}\label{teta_eq_3_minus_frac}
\theta_n(B_n)= 3-\frac{4}{n+1}.
\end{equation}
This is equivalent to (\ref{theta_1234_equalities}).
As it was noted in Introduction, if (\ref{teta_eq_3_minus_frac}) is true, then any projector having 
minimum norm corresponds to a regular simplex inscribed into $B_n$.
This completes the proof.
\end{proof}

\section{Concluding remarks}\label{nev_s6}

The inequality (\ref{norm_P_reg_ineqs}) of Theorem  \ref{norm_P_regular} 
 implies that $\theta_n(B_n)$ (the minimum norm of the projector 
for linear interpolation on the unit ball $\|x\|\leq 1$) satisfies the inequality
 \begin{equation}\label{theta_n_B_n_leq_sqrt_n_plus_1}
\theta_n(B_n)\leq \sqrt{n+1}.
\end{equation}
At least for those $n$, when  $\sqrt{n+1}$ is not integer,
i.\,e. $n\ne m^2-1$, this inequality is strict.
If a projector haves the minimal norm and its nodes are the vertices 
of a regular simplex inscribed into the sphere $\|x\|=1$, then 
\begin{equation}\label{theta_n_B_n_equality_psi_a}
\theta_n(B_n)=\max\{\psi(a),\psi(a+1)\}.
\end{equation}
The equality (\ref{theta_n_B_n_equality_psi_a}) holds, for example, 
for $1\leq n\leq 4$, when it is equivalent to (\ref{teta_eq_3_minus_frac}) 
(see Section~5).

Questions concerning accuracy of (\ref{theta_n_B_n_leq_sqrt_n_plus_1})
and correctness of (\ref{theta_n_B_n_equality_psi_a}) for $n>4$
will be discussed in a further paper.

\begin{figure}[h!]
 \centering
\includegraphics[width=15.5cm]{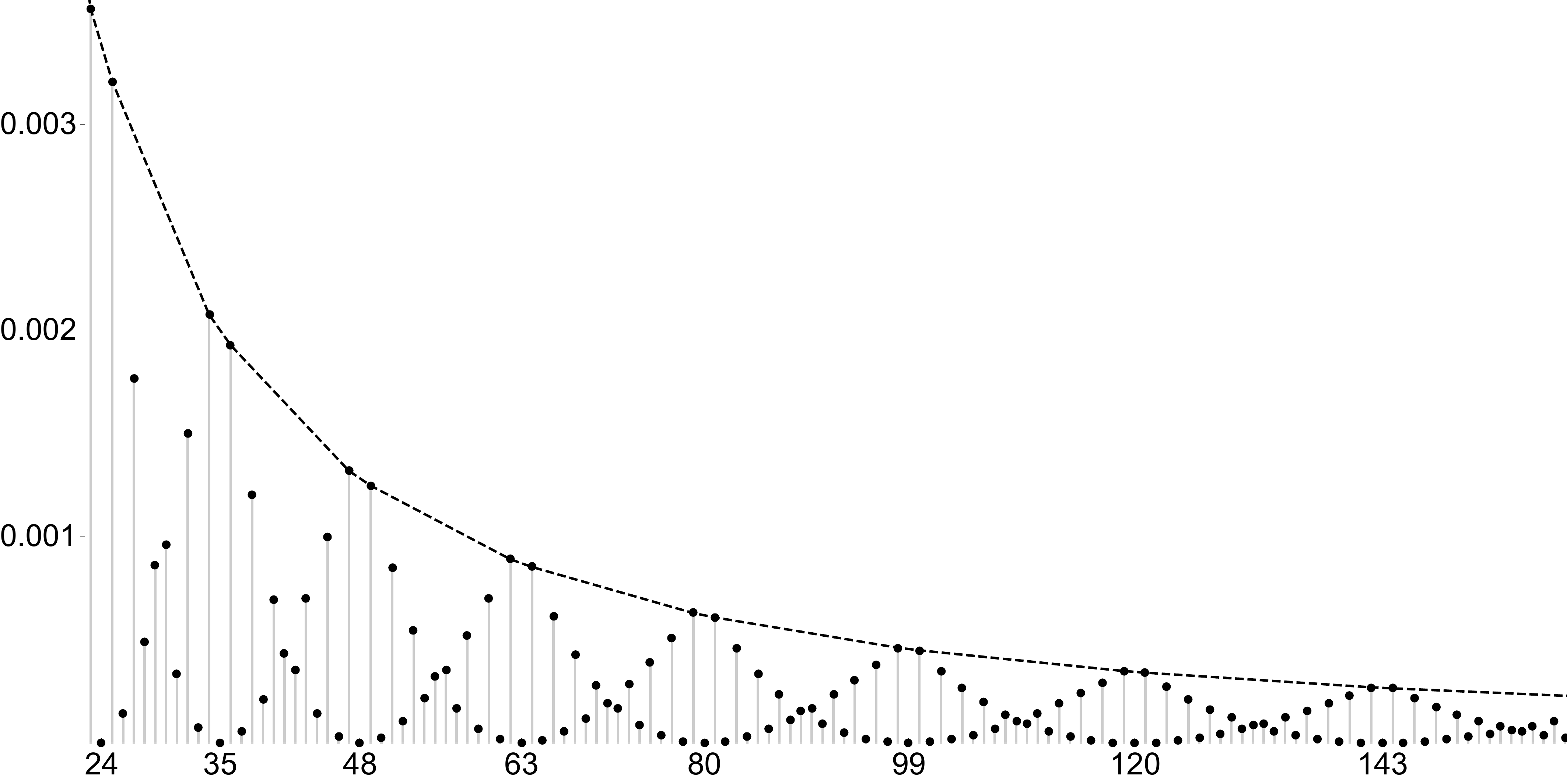}
\captionsetup{justification=centering}
\caption{The numbers $d_n=\sqrt{n+1}- \|P\|_{B_n}$ for $23\leq n\leq 160$}
\label{fig:nev_ukl_graph_diff}
\end{figure}

\begin{figure}[h!]
 \centering
\includegraphics[width=15.5cm]{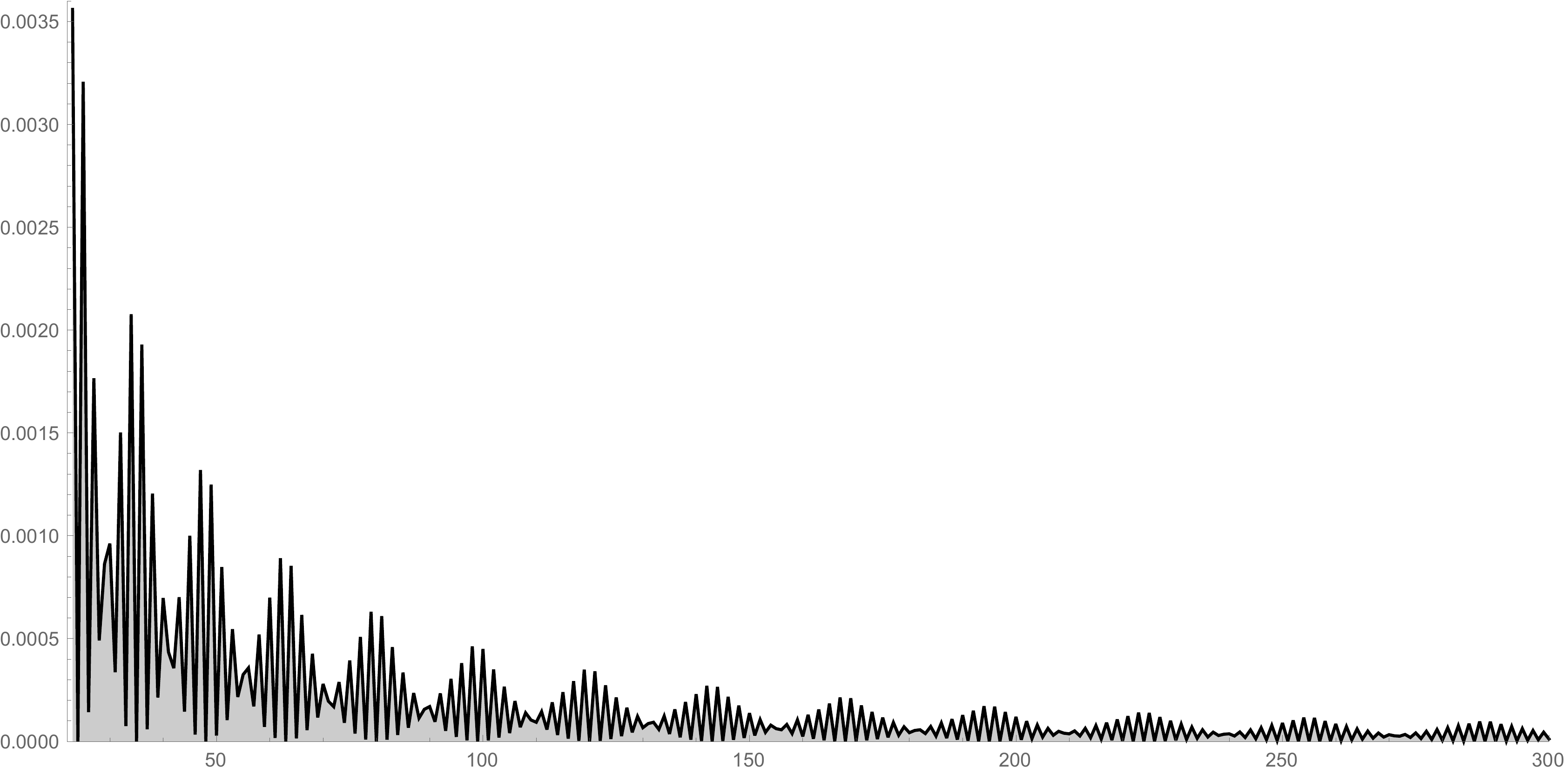}
\captionsetup{justification=centering}
\caption{The numbers $d_n=\sqrt{n+1}- \|P\|_{B_n}$ for $23\leq n\leq 300$}
\label{fig:nev_ukl_graph_diff_long}
\end{figure}

\begin{table}[h!]
\begin{center}
\captionsetup{justification=centering}
\caption{Norm of $P$ for a regular simplex inscribed into $B_n$}
\label{tab:nev_uhl_norm_est}
\bgroup
$
\def\arraystretch{1.7}
\begin{array}{|c|c|c|c|c|c|c|c|c|}
\hline
n & t_- & a & a+1 & \psi(a) & \psi(a+1) & k^* & \| P\|_{B_n}  & \| P\|_{B_n}  \\
\hline
 1 & 1-\frac{1}{\sqrt{2}} & 0 & 1 & 1 & 1 & 1 & 1 & 1 \\
 \hline
 2 & \frac{3-\sqrt{3}}{2}  & 0 & 1 & 1 & \frac{5}{3} & 1 & \frac{5}{3} & 1.6666\ldots \\
 \hline
 3 & 1 & 1 & 2 & 2 & \sqrt{3} & 1 & 2 & 2 \\
 \hline
 4 & \frac{ 5-\sqrt{5}}{2} & 1 & 2 & \frac{11}{5} & \frac{1+4 \sqrt{6}}{5}  & 1 & \frac{11}{5} & 2.2 \\
 \hline
 5 & 3-\sqrt{\frac{3}{2}} & 1 & 2 & \frac{7}{3} & \frac{1+2 \sqrt{10}}{3}  & 2 & \frac{1+2 \sqrt{10}}{3}  & 2.4415\ldots \\
 \hline
 6 & \frac{7-\sqrt{7}}{2}  & 2 & 3 & \frac{3+4 \sqrt{15}}{7}  & \frac{ 1+12 \sqrt{2}}{7} & 2 & 
 \frac{3+4 \sqrt{15}}{7}  & 2.6417\ldots \\
 \hline
 7 & 4-\sqrt{2} & 2 & 3 & \frac{1+\sqrt{21}}{2}  & \frac{1+\sqrt{105}}{4}  & 3 & \frac{1+\sqrt{105}}{4}  & 2.8117\ldots \\
 \hline
 8 & 3 & 3 & 4 & 3 & \frac{1+8 \sqrt{10}}{9}  & 3 & 3 & 3 \\
 \hline
 9 & 5-\sqrt{\frac{5}{2}} & 3 & 4 & \frac{2+3 \sqrt{21}}{5}  & \frac{1+6 \sqrt{6}}{5}  & 3 & \frac{2+3 \sqrt{21}}{5}  & 3.1495\ldots \\
 \hline
 10 & \frac{11-\sqrt{11}}{2}  & 3 & 4 & \frac{ 5+8 \sqrt{15}}{11} & \frac{ 3+4 \sqrt{70}}{11} & 4 & 
 \frac{3+4 \sqrt{70}}{11}  & 3.3151\ldots \\
 \hline
 11 & 6-\sqrt{3} & 4 & 5 & \frac{ 1+2 \sqrt{22}}{3} & \frac{ 1+\sqrt{385}}{6} & 4 & \frac{1+2 \sqrt{22}}{3}  & 3.4602\ldots \\
 \hline
 12 & \frac{13-\sqrt{13}}{2}  & 4 & 5 & \frac{5+24 \sqrt{3}}{13}  & \frac{3+8 \sqrt{30}}{13}  & 5 & 
 \frac{3+8 \sqrt{30}}{13}  & 3.6013\ldots \\
 \hline
 13 & 7-\sqrt{\frac{7}{2}} & 5 & 6 & \frac{2+3 \sqrt{65}}{7}  & \frac{1+4 \sqrt{39}}{7}  & 5 & 
 \frac{2+3 \sqrt{65}}{7}  & 3.7409\ldots \\
 \hline
 14 & \frac{15-\sqrt{15}}{2}  & 5 & 6 & \frac{1+4 \sqrt{7}}{3}  & \frac{1+4 \sqrt{21}}{5}  & 6 & 
 \frac{1+4 \sqrt{21}}{5}  & 3.8660\ldots \\
 \hline
 15 & 6 & 6 & 7 & 4 & \frac{1+3 \sqrt{105}}{8}  & 6 & 4 & 4 \\
 \hline
 50 & \frac{51-\sqrt{51}}{2}  & 21 & 22 & \frac{3+20 \sqrt{35}}{17}  & \frac{7+20 \sqrt{319}}{51}  & 22 & 
 \frac{7+20 \sqrt{319}}{51}  & 7.1414\ldots \\
 \hline
 100 & \frac{101-\sqrt{101}}{2}  & 45 & 46 & \frac{11+120 \sqrt{70}}{101}  & \frac{ 9+20 \sqrt{2530}}{101} & 45 & \frac{11+120 \sqrt{70}}{101}  & 10.0494\ldots \\
 \hline
 1000 & \frac{1001-\sqrt{1001}}{2}  & 484 & 485 & \frac{3+40 \sqrt{5170}}{91}  & \frac{31+200 \sqrt{25026}}{1001} & 485 & \frac{31+200 \sqrt{25026}}{1001} & 31.6385\ldots \\
 \hline
\end{array}
$
\egroup
\end{center}
\end{table}

In conclusion, we present some illustrations and the results of numerical analysis. 
Graphs of the function 
$$ \psi(t)=\frac{2\sqrt{n}}{n+1}\Bigl(t(n+1-t)\Bigr)^{1/2}+
\left|1-\frac{2t}{n+1}\right|, \quad  0\leq t\leq n+1,$$
for $n=3, 4, 5, 15$ are shown in Fig.~1--4.
The maximum points of of this function 
$t_-:=\frac{n+1}{2}-\frac{\sqrt{n+1}}{2},$  
$t_+:=\frac{n+1}{2}+\frac{\sqrt{n+1}}{2}$ 
are marked.
Points 
$a=\left\lfloor\frac{n+1}{2}-\frac{\sqrt{n+1}}{2}\right\rfloor$
and
$a+1$
are denoted. One of these two points maximizes 
$\psi(k)$ for integer $1\leq k\leq \frac{n+1}{2}$.
In~Fig.~5--6 the values of the difference $d_n=\sqrt{n+1}-\|P\|_{B_n}$ are presented
for $n>23$. Here $P$ is an interpolation projector corresponding to a regular simplex inscribed
into~$B_n$.
As we have established, $\|P\|_{B_n}\leq\sqrt{n+1}$ with the equality 
only for $n=3, 8, 15, 24, 35, 48, 63, 80,$ $\ldots$, i.\,e., for
dimensions of the form $ m ^ 2-1 $. 
It is at these points that $d_n=0$, as it can be seen in Fig.~5.
The dashed line denotes the graph of the linear interpolation spline 
$l$ constructed by the nodes $n=m^2-2,$ $n=m^2$
and by the values of $d_n$ in these nodes.
Always $d_n\leq l(n)$ and the equality is achieved only for $n=m^2-2$ and $n=m^2$.
The function $l(n)$ decreases and $l(n)\to 0$ as $n\to\infty$.
Thus, the following two-side estimate takes place:
$\sqrt{n+1}-l(n)\leq  \|P\|_{B_n} \leq \sqrt{n+1}$.
Both the right and the left inequalities turn into equalities for the infinite set of $n$'s.
This estimate is more accurate than (\ref{norm_P_reg_ineqs}).

Table 1 presents the values of $\|P\|_{B_n}$ for a regular simplex inscribed into the ball.
We use notations from Section~\ref{nev_s4}
The value $k^*$ is equal to the number of $-1$'s in the extremal set of
$f_j$ corresponding to the maximum in (\ref{P_norm_f_j_long}).
Since this maximum is equal to $\max\{\psi(a),\psi(a+1)\}$, the integer
$k^*$ is equal to that of numbers $a$ or $a+1$, for which $\psi(t)$
takes a larger value (see the proof of Theorem \ref{norm_P_regular}).
If $k^*=1$, then in $B_n$ there is an $1$-point with respect to 
the simplex $S$ corresponding to $P$.
For such $ n $,
\begin{equation}\label{nev_ksi_P_eq_regular}
\xi(B_n;S)=
\frac{n+1}{2}\Bigl( \|P\|_{B_n}-1\Bigr)+1
\end{equation}
(see (\ref{nev_ksi_P_ineq})).
Since $S$ is a regular simplex inscribed into $B_n$,  then
$\xi(B_n;S)=n$  and (\ref{nev_ksi_P_eq_regular})
is equivalent to $\|P\|_{B_n}=3-\frac{4}{n+1}$. 
However, starting from $n=5$, always $k^*>1$. 
Moreover, $k^*$  increases along with $n$.
This corresponds to the fact that (\ref{nev_ksi_P_eq_regular}) and the equality
$\|P\|_{B_n}=3-\frac{4}{n+1}$ take place only for $1\leq n\leq 4$.
Initially, this effect was discovered in the course of computer experiments. 
Later the analytical solution of the problem was  found.

If $n+1$ is an Hadamard number, the validity of equality similar to 
(\ref{nev_ksi_P_eq_regular}) can also be considered for a cube and 
the inscribed regular simplex.
It is interesting to note that an $1$-point of $Q_n$ with respect to such simplex
$S$ exists not only for $n=1$ and $n=3$, but also for $n=7$.
In the specified cases
$$
\xi(Q_n;S)=\frac{n+1}{2}\Bigl( \|P\|_{Q_n}-1\Bigr)+1,
$$
therefore, $\theta_3=3$ and $\theta_7=\frac{5}{2}$. 
This theme is discussed in detail in [4].

\bigskip
\centerline{\bf\Large References}

\begin{itemize}

\item[1.]   
Nevskij~M.\,V. Inequalities for the norms of interpolating projections, 
{\it Model.  Anal. Inform. Sist.}, 2008, vol.~15, no.~3. pp.~28--37 (in~Russian).

\item[2.]
Nevskij~M.\,V.  On a certain relation for the minimal norm of an interpolational projection,
{\it Model.  Anal. Inform. Sist.}, 2009, vol.~16, no.~1, pp.~24--43 (in~Russian).

\item[3.]
Nevskii~M.\,V. On a property of $n$-dimensional simplices,
{\it Math. Notes},   2010, vol.~87, no.~4, pp.~543--555.

\item[4.]
 Nevskii,~M.\,V.,
{\it Geometricheskie ocenki v polinomialnoi interpolyacii} 
(Geometric Estimates in Polynomial Interpolation), Yaroslavl': Yarosl. Gos. Univ., 2012 \linebreak
(in~Russian).

\item[5.]
 Nevskii~M.\,V., Ukhalov A.\,Yu. On numerical charasteristics of a simplex and their estimates,
{\it Model.  Anal. Inform. Sist.}, 2016, vol.~23, no.~5, pp.~603--619 (in~Russian).
English transl.: {\it Aut.
Control Comp. Sci.}, 2017, vol.~51, no.~7, pp.~757--769.

\item[6.]
Nevskii~M.\,V., Ukhalov A.\,Yu. New estimates of numerical values related to a simplex,
{\it Model.  Anal. Inform. Sist.}, 2017, vol.~24, no.~1, pp.~94--110 (in~Russian).
English transl.: {\it Aut.
Control Comp. Sci.}, 2017, vol.~51, no.~7, pp.~770--782.

\item[7.]
Nevskii~M.\,V., Ukhalov A.\,Yu. On $n$-Dimensional Simplices Satisfying Inclusions
$S\subset [0,1]^n\subset nS$,
{\it Model.  Anal. Inform. Sist.}, 2017, vol.~24, no.~5, pp.~578--595 (in~Russian).
English transl.: {\it Aut.
Control Comp. Sci.}, 2018, vol.~52, no.~7, pp.~667--679.

\item[8.]
Nevskii~M.\,V., Ukhalov A.\,Yu. On Minimal  Absorption Index for an $n$-Dimensional Simplex,
{\it Model.  Anal. Inform. Sist.}, 2018, vol.~25, no.~1, pp.~140--150 (in~Russian).
English transl.: {\it Aut.
Control Comp. Sci.}, 2018, vol.~52, no.~7, pp.~680--687.

\item[9.]
Nevskii~M.\,V., Ukhalov A.\,Yu. On optimal interpolation by linear functions on an $n$-dimensional cube,
{\it Model.  Anal. Inform. Sist.}, 2018, vol.~25, no.~3, pp.~291--311 (in~Russian).
English transl.: {\it Aut.
Control Comp. Sci.}, 2018, vol.~52, no.~7, pp.~828--842.

\item[10.]
Nevskii~M.\,V. On some problems for a simplex and a ball in ${\mathbb R}^n$,
{\it Model.  Anal. Inform. Sist.}, 2018, vol.~25, no.~6, pp.~681--692 (in~Russian).

\item[11.]
Hall~M., Jr, {\it Combinatorial theory}, Blaisdall publishing company,
Waltham \linebreak
(Massachusets) -- Toronto -- London, 1967.

\item[12.]
Hudelson,~M., Klee, V., and Larman,~D.,
Largest $j$-simplices in $d$-cubes: some relatives of the
Hadamard maximum determinant problem,
{\it Linear Algebra Appl.}, 1996, vol.~241--243, pp.~519--598.

\item[13.] Nevskii~M., Ukhalov A. Perfect simplices in ${\mathbb R}^5$,
{\it Beitr\"{a}ge zur Algebra und Geometrie~/ Contributions to Algebra and Geometry},
2018, vol.~59, no.~3, pp.~501--521.

\end{itemize}

\end{document}